\documentclass{amsart}
\usepackage{pdfpages}
\newtheorem{theorem}{Theorem}[section]
\newtheorem{lemma}[theorem]{Lemma}

\theoremstyle{definition}
\newtheorem{definition}[theorem]{Definition}

\newtheorem{proposition}{Proposition}
\theoremstyle{remark}

\numberwithin{equation}{section}

\begin{document}

\title{Mannheim curves with modified orthogonal frame in Euclidean 3-space}

\author{Mohamd Saleem Lone}
\address{International Centre for Theoretical Sciences, Tata Institute of Fundamental Research, 560089, Bengaluru, India}
\curraddr{International Centre for Theoretical Sciences, Tata Institute of Fundamental Research, 560089, Bengaluru, India}
\email{mohamdsaleem.lone@icts.res.in}

\author{Hasan Es}
\address{Gazi University, Gazi Educational Faculty, Department of
Mathematical Education, 06500 Teknikokullar / Ankara-Turkey}
\email{hasanes@gazi.edu.tr}

\author{Murat Kemal Karacan}
\address{Usak University, Faculty of Sciences and Arts, Department of
Mathematics,1 Eylul Campus, 64200,Usak-Turkey}
\email{murat.karacan@usak.edu.tr}

\author{Bahaddin Bukcu}
\address{Gazi Osman Pasa University, Faculty of Sciences and Arts,
Department of Mathematics,60250,Tokat-Turkey}
\email{bbukcu@yahoo.com}


%    General info
\subjclass[2000]{53A04, 53A35}

%\date{January 1, 2001 and, in revised form, June 22, 2001.}


\keywords{Mannheim curves, Mannheim partner curve, Modified orthogonal frame.}

\begin{abstract}
In this paper, we investigate Mannheim pairs, Frenet-Mannheim curves and Weakened
Mannheim curves with respect to the modified orthogonal frame in Euclidean 3-space$(E^3)$. We obtain some characterizations of these curves.
\end{abstract}

\maketitle

\section{\protect \bigskip Introduction}
In the study of the classical differential geometry of space curves, finding the corresponding relations between different space curves has been an important and interesting characterization problem of space curves. For example, if the normal vector of one space curve $\varphi$ is normal to another curve $\psi$, then $\psi$ is called as the Bertrand mate of $\varphi$. Liu [2] characterized similar type of curves - Weakened Bertrand curves and Frenet Bertrand curves under weakened conditions. There is another important class of space curves called Mannheim curves, where the normal vector of one curve is the binormal vector of some other curve, such a pair of curves is called as a Mannheim pair. Liu and Wang [3] derived the necessary and sufficient conditions for a curve to possess a Mannheim partner curve in Euclidean and Minkowski spaces. \"{O}ztekin and Erg\"{u}t [5] studied the null Mannheim curves in Minkowski space and derived some necessary and sufficient conditions. Recently, Tun\c{c}er et al.[8] obtained some characterization results of the non-null weakened Mannheim curves in Minkowski 3-space. Moreover, Karacan [1] characterized the Weakened Mannheim curves in Euclidean 3-space. As of now, authors usually have studied Mannheim curves with respect to the classical Frenet-Serret frame of a curve, where we are supposed to consider that the curvature $\kappa(s) \neq 0$.  In this paper, we shall drop the condition of $\kappa(s)\neq 0$ and consider a general set of curves with a discrete set of zeros of $\kappa(s)$ to characterize the Mannheim curves according to modified orthogonal frame in Euclidean 3-space. \bigskip \bigskip \ 

\section{Preliminaries}

\bigskip Let $\varphi(s)$ be a $C^{3}$ space curve in Euclidean 3-space $E^3$, parametrized by arc length $s$. We also assume that{ \it its
curvature $\kappa (s)\neq 0$ anywhere}. Then an
orthonormal frame $\left \{ t,n,b\right \} $ exists satisfying the
Frenet-Serret equations%
\begin{equation}
\left[ 
\begin{array}{c}
t^{\prime }(s) \\ 
n^{\prime }(s) \\ 
b^{\prime }(s)%
\end{array}%
\right] =\left[ 
\begin{array}{ccc}
0 & \kappa & 0 \\ 
-\kappa & 0 & \tau \\ 
0 & -\tau & 0%
\end{array}%
\right] \left[ 
\begin{array}{c}
t(s) \\ 
n(s) \\ 
b(s)%
\end{array}%
\right],  \tag{2.1}
\end{equation}%
where $t$ is the unit tangent, $n$ is the unit principal normal, $b$ is the unit binormal, and $\tau (s)$ is the torsion. For a given  $C^{1}$ function $\kappa (s)$  and a continuous function $\tau (s)$, there exists a $C^{3}$ curve $\varphi$ which has an orthonormal frame $\left \{t,n,b\right \} $ satisfying the Frenet-Serret frame (2.1). Moreover, any other curve $\tilde{\varphi}$ satisfying the same conditions, differs from $\varphi$ only by a rigid motion.

Now let $\varphi(t)$ be a general analytic curve which can be reparametrized by its arc length. Assuming that the curvature function has {\it discrete zero points} or $\kappa(s)$ {\it is not identically zero}, we have an orthogonal frame $\left \{ T,N,B\right \} $ defined as follows:%
\begin{equation*}
T=\frac{d\varphi }{ds},\quad  N=\frac{dT}{ds},\quad B=T\times N,
\end{equation*}%
where $T\times N$ is the vector product of $T$ and $N$. The relations
between $\left \{ T,N,B\right \} $ and previous Frenet frame vectors at
non-zero points of $\kappa $ are%
\begin{equation}
T=t,N=\kappa n,B=\kappa b.  \tag{2.2}
\end{equation}%
Thus, we see that $N(s_{0})=B(s_{0})=0$ when $\kappa (s_{0})=0$ and squares of the length
of $N$ and $B$ vary analytically in $s$. From Eq. (2.2), it is easy to calculate

\begin{equation}
\left[ 
\begin{array}{c}
T^{\prime }(s) \\ 
N^{\prime }(s) \\ 
B^{\prime }(s)%
\end{array}%
\right] =\left[ 
\begin{array}{ccc}
0 & 1 & 0 \\ 
-\kappa ^{2} & \frac{\kappa ^{\prime }}{\kappa } & \tau \\ 
0 & -\tau & \frac{\kappa ^{\prime }}{\kappa }%
\end{array}%
\right] \left[ 
\begin{array}{c}
T(s) \\ 
N(s) \\ 
B(s)%
\end{array}%
\right]  \tag{2.3}
\end{equation}%
and 
\begin{equation*}
\tau =\tau (s)=\frac{\det \left( \varphi ^{\prime },\varphi ^{\prime \prime
},\varphi ^{\prime \prime \prime }\right) }{\kappa ^{2}}
\end{equation*}%
is the torsion of $\varphi $. From Frenet-Serret equations, we know that any point, where $\kappa ^{2}=0$ is a removable singularity of $\tau $. Let $\left \langle ,\right \rangle $ be the standard inner product of $E^3$, then $\left \{ T,N,B\right \} $ satisfies:%
\begin{equation}
\left \langle T,T\right \rangle =1,\left \langle N,N\right \rangle =\left
\langle B,B\right \rangle =\kappa ^{2},\left \langle T,N\right \rangle
=\left \langle T,B\right \rangle =\left \langle N,B\right \rangle =0
\tag{2.4}.
\end{equation}%
The orthogonal frame defined in Eq. (2.3) satisfying Eq. (2.4) is called as modified orthogonal frame[7].
\begin{definition}
Let $\varphi$ and $\psi$ be two space curves in Euclidean 3-space, such that there exists a corresponding relation between $\varphi$ and $\psi$ in such a way that the principal normal vectors of $\varphi$ coincides with the binormal
vectors of $\psi$ at the corresponding points. Then $\psi$ is called as a Mannheim partner curve of Mannheim curve $\varphi$ and the pair $\left \{ \varphi, \psi\right \} $
is called as a Mannheim pair[3].
\end{definition}
\begin{definition}
A regular Mannheim curve $\psi(s^ \ast)$ with non-vanishing curvature for which there exists some other 
regular curve $\varphi(s)$, parametrized by arc length and has non-vanishing curvature, in
bijection with it in such a way that the binormal to $\varphi(s)$ and the principal normal to $\psi(s^\ast)$
at each corresponding pair of points coincide with the joining line of corresponding points. The curve $\varphi(s)$ is said to be a Mannheim conjugate of $\psi(s^\ast)$[1,6].
\end{definition}
\begin{definition}
A Frenet-Mannheim(FM) curve is a Frenet curve $\psi(s^\ast)$ for which we have some other Frenet curve $\varphi(s)$ and $\varphi^\prime(s)$ is non-vanishing, in bijection with it so that, by appropriate selection of the Frenet frame,
the binormal vector $B_{\varphi}(s)$ and the principal normal vector $N_{\psi}(s^\ast)$  both lie on the joining line of the corresponding points of $\varphi(s)$ and  $\psi(s^\ast)$. The curve $\varphi(s)$ is called FM conjugate of $\psi(s^\ast)$[1,6].  
\end{definition}
\begin{definition}\label{def 1}
A Weakened Mannheim(WM) curve $\psi(s^\ast)$ is a regular curve such that there exists some other regular curve $\varphi(s)$ and a homeomorphism $\rho:I\rightarrow I^\ast$ such that:
\begin{itemize}
\item[(i)] There exists two closed subsets $M$, $N$  of $I$ which are disjoint  with empty interiors such that $\rho \in C^\infty$ on $I\backslash N$, $\frac{ds^\ast}{ds}=0$ on $M$, $\rho^{-1} \in C^\infty$ on $\rho\left(I \backslash M\right)$ and $\frac{ds}{ds^\ast}=0$ on $\rho(N)$, 
\item[(ii)] The joining line of $s$ and $s^\ast$ of $\varphi(s)$ and $\psi(s^\ast)$, respectively is orthogonal to $\varphi(s)$ and $\psi(s^\ast)$ at the corresponding points $s,s^\ast$ respectively, and is along the principal normal to $\varphi(s)$ or $\psi(s^\ast)$, whenever it is well defined. The curve $\varphi(s)$  is said to be a WM conjugate of $\psi(s^\ast)$[1,6].
\end{itemize}
\end{definition}
From the classical differential geometry, we found that there is a rich literature available on the Bertrand pairs in comparison to the Mannheim curves. Thus, in this paper, we study the Mannheim curves according to the modified orthogonal frame in Euclidean 3-space and obtain several conditions for Mannheim partner, $FM$ and $WM$
curves.

\section{Mannheim partner curves according to modified orthogonal frame in $%
E^{3}$}

\begin{theorem}
Let $C:\varphi (s)$ be a Mannheim curve in $E^{3}$ parameterized by its
arc length $s$ and let $C^{\ast }:\psi (s^{\ast })$ be the Mannheim
partner curve of $C$ with an arc length parameter $s^{\ast }$. The distance
between corresponding points of the Mannheim partner curves in $E^{3}$ is $%
\left \vert c\right \vert \kappa _{\varphi }$, where $c$ is nonzero constant
and $\kappa _{\varphi }$ is the curvature of curve $\varphi $.
\end{theorem}

\begin{proof}
From the definition of Mannheim pair $\{C,C^{\ast }\}$, we can write $\overrightarrow{%
\varphi (s)\psi (s^{\ast })}=\mu (s)N_{\varphi }(s)$, or 
\begin{equation}
\psi (s^{\ast })=\varphi (s)+\mu (s)N_{\varphi }(s)  \tag{3.1}
\end{equation}%
for some function $\mu $ $\left( s\right)$. Taking derivative with
respect to $s$ and using Eqn. (2.3), we get

\begin{equation*}
\psi ^{^{\prime }}(s^{\ast })=T_{\varphi }+\mu ^{^{\prime }}N_{\varphi }+\mu
(-\kappa _{\varphi }^{2}T_{\varphi }+\frac{\kappa _{\varphi }^{^{\prime }}}{%
\kappa _{\varphi }}N_{\varphi }+\tau _{\varphi }B_{\varphi })
\end{equation*}%
or 
\begin{equation}
T_{\psi }\frac{ds^{\ast }}{ds}=\left( 1-\mu \kappa _{\varphi }^{2}\right)
T_{\varphi }+\left( \mu ^{^{\prime }}+\mu \frac{\kappa _{\varphi }^{^{\prime }}%
}{\kappa _{\varphi }}\right) N_{\varphi }+\mu \tau _{\varphi }B_{\varphi }. 
\tag{3.2}
\end{equation}%
Taking inner product of Eqn. (3.2) with $B_{\psi }$ and
consider $\frac{N_{\varphi }}{\kappa _{\varphi }}=\epsilon \frac{%
B_{\psi }}{\kappa _{\psi }}\left( \epsilon =\pm 1\right) ,$ we get 
\begin{equation}
\mu ^{^{\prime }}+\mu \frac{\kappa _{\varphi }^{\prime }}{\kappa _{\varphi }}=0%
\text{ or }\mu =\frac{c}{\kappa _{\varphi }}.  \tag{3.3}
\end{equation}%
This means that $\mu $ is not constant except $c=0$. On the other hand, from
the distance function between the points of $\varphi (s)$ and $\psi (s^{\ast })$,
we have%
\begin{equation*}
d\left( \varphi (s),\psi (s^{\ast })\right) =\left \vert c\right \vert
\kappa _{\varphi }(s)\text{.}
\end{equation*}
\end{proof}

\begin{theorem}
A space curve $\varphi (s)$ in $E^{3}$ with respect to modified orthogonal frame is a Mannheim curve if and only if its curvature $%
\kappa _{\varphi }$ and torsion $\tau _{\varphi }$ satisfies:
\begin{equation}
\kappa _{\varphi }=c\left( \kappa _{\varphi }^{2}+\tau _{\varphi }^{2}\right) ,
\tag{3.4}
\end{equation}%
where $c$ is non-zero constant.

\begin{proof}
Let $C:\varphi (s)$ be a Mannheim curve in $E^{3}$ with arc length
parameter $s$ and $C^{\ast}:\psi(s^{_{\ast}})$ the Mannheim partner
curve of $C$ with arc length parameter $s^{\ast }$. Inserting Eqn. (3.3) in
Eqn. (3.1), we get 
\begin{equation}
\psi (s^{\ast })=\varphi (s)+\frac{c}{\kappa _{\varphi }(s)}N_{\varphi }(s) 
\tag{3.5}
\end{equation}%
for some non-zero constant $c$. Differentiating Eqn. (3.5) with respect to $s$ and applying the
modified orthogonal frame formulas, we obtain 
\begin{equation}
T_{\psi }\frac{ds^{\ast }}{ds}=\left( 1-c\kappa _{\varphi }\right) T_{\varphi
}+\frac{c\tau _{\varphi }}{\kappa _{\varphi }}B_{\varphi }.  \tag{3.6}
\end{equation}%
Again differentiating Eqn. (3.6) with respect to $s$ and
applying the modified orthogonal frame formulas, we get$\allowbreak $ 
\begin{eqnarray*}
N_{\psi }\left(\frac{ds^{\ast }}{ds}\right)^2+T_{\psi}\frac{d^2 s^{\ast}}{ds^2} &=&-c\kappa _{\varphi }^{\prime
}T_{\varphi }+\left( 1-c\kappa _{\varphi }\right) N_{\varphi }+\frac{c\tau
_{\varphi }^{\prime }\kappa _{\varphi }-c\kappa _{\varphi }^{\prime }\tau
_{\varphi }}{\kappa _{\varphi }^{2}}B_{\varphi } \\
&&+\frac{c\tau _{\varphi }}{\kappa _{\varphi }}\left( -\tau _{\varphi
}N_{\varphi }+\frac{\kappa _{\varphi }^{\prime }}{\kappa _{\varphi }}B_{\varphi
}\right)
\end{eqnarray*}%
or 
\begin{equation}
N_{\psi }\left(\frac{ds^{\ast }}{ds}\right)^2+T_{\psi}\frac{d^2 s^{\ast}}{ds^2}=-c\kappa _{\varphi }^{\prime }T+%
\frac{1}{\kappa _{\varphi }}\left( \kappa _{\varphi }-c\kappa _{\varphi
}^{2}-c\tau _{\varphi }^{2}\right) N_{\varphi }+c\tau _{\varphi }^{\prime
}\kappa _{\varphi }B_{\varphi }.  \tag{3.7}
\end{equation}%
Taking the inner product of the Eqn. (3.7) with $B_{\psi }$, we get%
\begin{equation}
\kappa _{\varphi }-c\kappa _{\varphi }^{2}-c\tau _{\varphi }^{2}=0\text{ or }%
\kappa _{\varphi }=c\left( \kappa _{\varphi }^{2}+\tau _{\varphi }^{2}\right) .%
\text{ }  \tag{3.8}
\end{equation}%
This completes the proof.

Conversely, if the curvature $\kappa _{\varphi }$ and the torsion $\tau _{\varphi }$
of the curve $C$ satisfy Eqn. (3.4) for some nonzero constant $c,$ then
define a curve $C$ by Eqn. (3.5) and we will prove that $C$ is Mannheim and 
$C^{\ast }$ is the partner curve of $C.$ We already have found  the
equality below 
\begin{equation*}
T_{\psi }\frac{ds^{\ast }}{ds}=\left( 1-c\kappa _{\varphi }\right) T_{\varphi
}+\frac{c\tau _{\varphi }}{\kappa _{\varphi }}B_{\varphi }.
\end{equation*}%
Differentiating last equality with respect to $s$ and
with the help of Eqn. (3.8), we get 
\begin{equation}
N_{\psi }\left( \frac{ds^{\ast }}{ds}\right) ^{2}+T_{\psi }\frac{d^{2}s^{\ast}}{%
ds^{2}}=-c\kappa _{\varphi }^{\prime }\allowbreak T_{\varphi }+\frac{%
c\allowbreak \tau _{\varphi }^{\prime }}{\kappa _{\varphi }}B_{\varphi }. 
\tag{3.9}
\end{equation}%
Taking the cross product of Eqn. (3.6) with Eqn. (3.9), we obtain%
\begin{equation*}
\frac{ds}{ds^{\ast }}T_{\psi }\times \left[ N_{\psi }\left( \frac{ds^{\ast
}}{ds}\right) ^{2}+T_{\psi }\frac{d^{2}s^{\ast 2}}{ds}\right] =\left[
\left( 1-c\kappa _{\varphi }\right) T_{\varphi }+\frac{c\tau _{\varphi }}{%
\kappa _{\varphi }}B_{\varphi }\right] \times \left( -c\kappa _{\varphi
}^{\prime }\allowbreak T_{\varphi }+\frac{c\allowbreak \tau _{\varphi
}^{\prime }}{\kappa _{\varphi }}B_{\varphi }\right)
\end{equation*}%
or%
\begin{equation}
\left( \frac{ds^{\ast }}{ds}\right) ^{3}B_{\psi }=c\left( \allowbreak -\tau
_{\varphi }^{\prime }+c\tau _{\varphi }^{\prime }\kappa _{\varphi }-c\kappa
_{\varphi }^{\prime }\tau _{\varphi }\right) \frac{N_{\varphi }}{\kappa
_{\varphi }}.  \tag{3.10}
\end{equation}%
Since both $\frac{N_{\varphi }}{\kappa _{\varphi }}$ and $\frac{B_{\psi
}}{\kappa _{\psi }}$ have unit lenght, we get%
\begin{equation}
\left( \frac{ds^{\ast }}{ds}\right) ^{3}=\frac{c\left( \allowbreak -\tau
_{\varphi }^{\prime }+c\tau _{\varphi }^{\prime }\kappa _{\varphi }-c\kappa
_{\varphi }^{\prime }\tau _{\varphi }\right) }{\kappa _{\psi }}.  \tag{3.11}
\end{equation}%
Thus, we have 
\begin{equation*}
\frac{B_{\psi }}{\kappa _{\psi }}=\epsilon \frac{N_{\varphi
}}{\kappa _{\varphi }},\epsilon =\pm 1
\end{equation*}%
or $N_{\varphi }$ and $B_{\psi }$ are linearly dependent. This completes the
proof.
\end{proof}
\end{theorem}

\begin{theorem}
A pair of curves $(C,C^{\ast })$ is a Mannheim pair iff \ the curvature $\kappa _{\psi }$
and the torsion $\tau _{\psi }$ of curve $C^{\ast }$ satisfy: 
\begin{equation}
\tau _{\psi }^{\prime }=\frac{d\tau _{\psi }}{ds^{\ast }}=\frac{\kappa
_{\psi }}{a}\left( 1+a^{2}\tau _{\psi }^{2}\right), \tag{3.12}
\end{equation}%
where $a$ is a non-zero constant.
\end{theorem}

\begin{proof}
Suppose that $C:\varphi (s)$ is a Mannheim curve. By the definition of $\varphi{(s)}$,
we may write 
\begin{equation}
\varphi (s)=\psi (s^{\ast })+\delta (s^{\ast })B_{\psi }(s^{\ast }) 
\tag{3.13}
\end{equation}%
for some function $\delta (s^{\ast })$. Differentiating
Eqn. (3.13) with respect to $s^{\ast }$, we get

\begin{equation}
T_{\varphi }\frac{ds}{ds^{\ast }}=T_{\psi }-\delta \tau _{\psi }N_{\psi
}+\left( \delta ^{\prime }+\delta \frac{\kappa _{\psi }^{\prime }}{\kappa
_{\psi }}\right) B_{\psi }.  \tag{3.14}
\end{equation}%
Since $N_{\varphi }$ and $B_{\psi }$ are linearly dependent, we get $\ $%
\begin{equation}
\text{ }\delta ^{\prime }+\delta \frac{\kappa _{\psi }^{\prime }}{\kappa
_{\psi }}=0\  \text{or }\delta (s^{\ast })=\frac{a}{\kappa _{_{\psi }}}. 
\tag{3.15}
\end{equation}%
This means that $\delta (s^{\ast })$ is not a constant for each $s^{\ast }$
except $a=0$. Thus, with the help of Eqn. (3.15), we can rewrite  Eqn. (3.14) as
follows 
\begin{equation}
T_{\varphi }\frac{ds}{ds^{\ast }}=T_{\psi }-\frac{a\tau _{_{\psi }}}{\kappa
_{_{\psi }}}N_{\psi }.  \tag{3.16}
\end{equation}

Let $\theta $ be the angle between $T_{\varphi }$ and $T_{\psi }$ at the corresponding points of $C$ and $C^{\ast }$ in Eqn. (3.13). Then taking the inner product of Eqn. (3.16) with $T_{\psi }$ and considering the
equality $\cos ^{2}\theta +\sin ^{2}\theta =1,$ we get%
\begin{equation}
\cos \theta =\frac{ds^{\ast }}{ds}  \tag{3.17}
\end{equation}%
and%
\begin{equation}
\frac{ds^{\ast }}{ds}=-a\tau _{\psi }\sin \theta .  \tag{3.18}
\end{equation}%
From Eqn. (3.17) and Eqn. (3.18), we find 
\begin{equation}
\frac{ds}{ds^{\ast }}=\frac{1}{\cos \theta }=-\frac{a\tau _{\psi }}{\sin
\theta }  \tag{3.19}
\end{equation}%
and 
\begin{equation}
\tan \theta =-a\tau _{_{\psi }}.  \tag{3.20}
\end{equation}%
Thus, we can write Eqn. (3.16) as follows 
\begin{equation}
T_{\varphi }=\left( \cos \theta \right) T_{\psi }+\frac{\sin \theta }{\kappa
_{\psi }}N_{\psi }.  \tag{3.21}
\end{equation}%
Differentiating Eqn. (3.21) with respect to $s^{\ast }$, we
get
\begin{equation}
N_{\varphi }\frac{ds}{ds^{\ast }}=-\sin \theta \left( \kappa _{\psi }+\theta
^{\prime }\right) T_{\psi }+\frac{\cos \theta }{\kappa _{\psi }}\left(
\kappa _{\psi }+\theta ^{\prime }\allowbreak \right) N_{\psi }+\left( 
\frac{\tau _{\psi }}{\kappa _{\psi }}\sin \theta \right) B_{\psi }. 
\tag{3.22}
\end{equation}%
From this equation and the fact that the direction of $\frac{N_{\varphi }}{%
\kappa _{\varphi }}$ is coincident with $\frac{B_{\psi }}{\kappa _{\psi }}$%
, we get%
\begin{equation}
\left \{ 
\begin{array}{c}
-\sin \theta \left( \kappa _{\psi }+\theta ^{\prime }\right) =0 \\ 
\frac{\cos \theta }{\kappa _{\psi }}\left( \kappa _{\psi }+\theta ^{\prime
}\allowbreak \right) =0%
\end{array}%
\right.  \tag{3.23}
\end{equation}%
or%
\begin{equation}
\theta ^{\prime }=-\kappa _{_{\psi }}.  \tag{3.24}
\end{equation}

Differentiating Eqn. (3.20) with respect to $s^{\ast }$ and
applying Eqn. (3.24), we get%
\begin{equation*}
\kappa _{\psi }+a^{2}\kappa _{\psi }\tau _{\psi }^{2}-a\tau _{\psi
}^{\prime }=0
\end{equation*}%
or%
\begin{equation*}
\tau _{\psi }^{\prime }=\frac{\kappa _{\psi }}{a}\left( 1+a^{2}\tau
_{\psi }^{2}\right) .
\end{equation*}%
Conversely, if the curvature $\kappa _{\psi }$ and the torsion $\tau _{\psi }$
of $C^{\ast }$ satisfy Eqn. (3.12) for some nonzero constant $a,$
then define a curve $C$ by Eqn. (3.13) and we will prove that $C$ is a
Mannheim and $C^{\ast }$ is the partner curve of $C.$ We can easily reduce
Eqn. (3.13) in the following expression 
\begin{equation*}
T_{\varphi }\frac{ds}{ds^{\ast }}=T_{\psi }-\frac{a\tau _{\psi }}{\kappa
_{\psi }}N_{\psi }.
\end{equation*}%
Differentiating above equality with respect to $s^{\ast }$ and with the
help of Eqn. (2.3), we get%
\begin{equation*}
N_{\varphi }\left( \frac{ds}{ds^{\ast }}\right) ^{2}+T_{\varphi }\frac{d^{2}s}{%
ds^{\ast 2}}=N_{\psi }+\frac{a\kappa _{\psi }^{\prime }\tau _{\psi
}-a\kappa _{\psi }\tau _{\psi }^{^{\prime }}}{\kappa _{\psi }^{2}}%
N_{\psi }-\frac{a\tau _{\psi }}{\kappa _{\psi }}\left( -\kappa _{\psi
}^{2}T_{\psi }+\frac{\kappa _{\psi }^{\prime }}{\kappa _{\psi }}N_{\psi
}+\tau _{\psi }B_{\psi }\right)
\end{equation*}%
or noticing Eqn. (3.12), we get 
\begin{equation}
N_{\varphi }\left( \frac{ds}{ds^{\ast }}\right) ^{2}+T_{\varphi }\frac{d^{2}s}{%
ds^{\ast 2}}=a\tau _{\psi }\kappa _{\psi }T_{\psi }-a^{2}\tau _{\psi
}^{2}N_{\psi }+\frac{c\tau _{\psi }^{2}}{\kappa _{\psi }}B_{\psi }. 
\tag{3.25}
\end{equation}%
Taking the cross product of Eqn. (3.16) with Eqn. (3.25), we have%
\begin{equation*}
\frac{ds}{ds^{\ast }}T_{\varphi }\times \left[ \left( \frac{ds}{ds^{\ast }}%
\right) ^{2}N_{\varphi }+T_{\varphi }\frac{d^{2}s}{ds^{\ast 2}}\right] =\left(
T_{\psi }-\frac{a\tau _{\psi }}{\kappa _{\psi }}N_{\psi }\right) \times
\left( a\tau _{\psi }\kappa _{\psi }T_{\psi }-a^{2}\tau _{\psi
}^{2}N_{\psi }+\frac{a\tau _{\psi }^{2}}{\kappa _{\psi }}B_{\psi }\right)
\end{equation*}%
or%
\begin{equation}
\left( \frac{ds}{ds^{\ast }}\right) ^{3}B_{\varphi }=-\frac{a\tau _{\psi
}^{2}}{\kappa _{\psi }}\left( \frac{a\tau _{\psi }}{\kappa _{\psi }}%
T_{\psi }+N_{\psi }\right) .  \tag{3.26}
\end{equation}%
Again taking the cross product of Eqn. (3.16) with Eqn. (3.26), we 
obtain 
\begin{equation*}
\left( \frac{ds}{ds^{\ast }}\right) ^{4}N_{\varphi }=\frac{a\tau _{\psi }^{2}%
}{\kappa _{\psi }}\left( \kappa _{\psi }^{2}+a^{2}\tau _{\psi
}^{2}\right) B_{\psi }
\end{equation*}%
or%
\begin{equation*}
N_{\varphi }=\frac{\kappa _{\varphi }}{\kappa _{\psi }}B_{\psi }.
\end{equation*}%
This means that the principal normal direction $\frac{N_{\varphi }}{\kappa
_{\varphi }}$ of $C:\varphi(s)$ coincides with the binormal direction 
$\frac{B_{\psi}}{\kappa_{\psi }}$of $C^{\ast}:\psi(s^{\ast }).$
Hence $C :\varphi (s)$ is a Mannheim curve and $C^{\ast } :\psi
(s^{\ast })$ is its Mannheim partner curve. Therefore, for each Mannheim
curve, there is a unique Mannheim partner curve.
\end{proof}

\begin{proposition}
A simple parametric transformation reduces the condition
\begin{equation*}
\tau _{\psi }^{\prime }=\frac{\kappa _{\psi }}{a}\left( 1+a^{2}\tau
_{\psi }^{2}\right) 
\end{equation*}%
to
\begin{equation*}
\tau _{\psi } = \frac{1}{a} \tan \left( \int \kappa_{\psi }ds+c_{0}\right) .
\end{equation*}%
Thus the existence of Mannheim partner curve to a Mannheim curve is unique.
\end{proposition}

\begin{proposition}
Let $\{\varphi (s),\psi(s^\ast)\}$ be a Mannheim pair, where both $\varphi$ and $\psi$ are parametrized by arc length $s$ and $s^\ast$, respectively. If $\varphi (s)$ is
a generalized helix according to modified frame in $E^{3}$, then
$\psi (s^{\ast })$ is a straight line.

\begin{proof}
Let $T_{\varphi },$ $ N_{\varphi }$, $B_{\varphi }$ be the tangent, the principal
normal and the binormal vectors of $\varphi (s)$, respectively.
From the definition of the Mannheim curve and properties of generalized helices, we have \
\begin{equation*}
\left \langle N_{\varphi },u\right \rangle =\left \langle B_{\psi
},u\right \rangle =0,
\end{equation*}%
\ where $u$ is some constant vector. Differentiating the last equality,
we derive 
\begin{equation*}
\left \langle -\tau _{\psi }N_{\psi }+\frac{\kappa ^{\prime }}{\kappa }%
B_{\psi },u\right \rangle =-\tau _{\psi }\left \langle N_{\psi
},u\right \rangle +\frac{\kappa ^{\prime }}{\kappa }\left \langle B_{\psi
},u\right \rangle =-\tau _{\psi }\left \langle N_{\psi },u\right \rangle =0.
\end{equation*}%
Since $\left \langle N_{\psi },u\right \rangle \neq 0,$ we get 
\begin{equation*}
\tau _{\psi }=0.
\end{equation*}%
Thus using the last equality in Eqn. (3.12), we easily obtain%
\begin{equation*}
\kappa _{\psi }=0.
\end{equation*}
\end{proof}
\end{proposition}

\section{\protect \bigskip Frenet Mannheim curves}

In this section, we characterize $FM$ curves, for that we begin with a lemma.

\begin{lemma}\label{lem1}
Suppose $\psi (s^{\star })$, $s^{\star }\in I^{\star }$ be a $FM$ curve with a FM conjugate $\varphi(s)$. We mark all the quantities of $\psi (s^{\star })$ with a asterisk and suppose
\begin{equation}
\psi (s^{\star })=\varphi (s)+\delta (s)B_{\varphi }(s). \tag{4.1}
\end{equation}%
Then the distance$(\left \vert \delta \right \vert)$ between corresponding
points of $\varphi (s)$, $\psi (s^{\star })$ is not constant i.e., $\delta =$
$c\kappa _{\varphi }$, $c\in R$ and $\left \langle T_{\varphi },T_{\psi }\right \rangle =\cos \theta $, where $\theta$ is a constant angle and
\begin{equation*}
\begin{array}{cc}
(i) & \sin \theta =-a\tau _{\varphi }\cos \theta \  \\ 
(ii) & \left( 1+\epsilon a\kappa _{\psi }\right) \sin \theta =a\tau _{\psi
}\cos \theta \\ 
(iii) & \cos ^{2}\theta =1+\epsilon a\kappa _{\psi }\  \  \  \\ 
(iv) & \sin ^{2}\theta =a^{2}\tau _{\varphi }\tau _{\psi }.\  \ 
\end{array}%
\end{equation*}
\end{lemma}

\begin{proof}
\bigskip \ From Eqn. (4.1), we have 
\begin{equation*}
\delta (s)=\left \langle \psi (s^{\star })-\varphi (s),B_{\varphi }(s)\right
\rangle,
\end{equation*}%
where $\delta(s)$ is of class $C^{\infty }$. Differentiating Eqn. (4.1) with respect to $s$, we get
\begin{equation*}
T_{\psi }\frac{ds^{\star }}{ds}=T_{\varphi }+\delta ^{\prime }B_{\varphi
}+\delta (-\tau _{\varphi }N_{\varphi }+\frac{\kappa ^{\prime }}{\kappa }%
B_{\varphi })
\end{equation*}%
or%
\begin{equation}
T_{\psi }\frac{ds^{\star }}{ds}=T_{\varphi }-\tau _{\varphi }\delta
N_{\varphi }+\left( \delta ^{\prime }+\delta \frac{\kappa _{\varphi
}^{\prime }}{\kappa _{\varphi }}\right) B_{\varphi }.  \tag{4.2}
\end{equation}%
By the given conditions, we have $B_{\varphi }=\epsilon N_{\psi }$ with $\epsilon
=\pm 1$. Taking the scalar multiplication of Eqn. (4.2) with $B_{\varphi }$, we obtain
\begin{equation*}
\frac{\delta ^{\prime }}{\delta }=-\frac{\kappa _{\varphi }^{\prime }}{%
\kappa _{\varphi }}\Rightarrow \delta =\frac{a}{\kappa _{\varphi }},a\in R.
\end{equation*}%
Therefore%
\begin{equation}
T_{\psi }\frac{ds^{\star }}{ds}=T_{\varphi }-\frac{a\tau _{\varphi }}{\kappa
_{\varphi }}N_{\varphi }.  \tag{4.3}
\end{equation}%
Now by the definition of $FM$ curve, we have $\frac{ds^{\star }}{ds}\neq 0$, so
that $T_{\psi }$ is a functions of $s$ of class $C^{\infty }$. Hence%
\begin{equation*}
\left \langle T_{\varphi },T_{\psi }\right \rangle ^{\prime }=\left \langle
N_{\varphi },T_{\psi }\right \rangle +\left \langle T_{\varphi },N_{\psi
}\right \rangle \frac{ds^{\star }}{ds}=0.
\end{equation*}%
This implies that $\left \langle T_{\varphi },T_{\psi }\right \rangle $ is
constant, thus there exists a constant angle $\theta$, such that%
\begin{equation}
T_{\psi }=T_{\varphi }\cos \theta +N_{\varphi }\frac{\sin \theta }{\kappa
_{\varphi }}.  \tag{4.4}
\end{equation}%
From Eqn. (4.3) and Eqn. (4.4), we get 
\begin{equation*}
\left( \frac{ds}{ds^{\star }}-\cos \theta \right) T_{\psi }-\left( \frac{ds%
}{ds^{\star }}\frac{a.\tau _{\varphi }}{\kappa _{\varphi }}+\frac{\sin \theta 
}{\kappa _{\varphi }}\right) N_{\psi }=0.
\end{equation*}%
Since $T_{\psi }$ and $N_{\psi }$ are linearly independent vectors, we have 
\begin{equation}
\frac{ds}{ds^{\star }}=\cos \theta   \tag{4.5}
\end{equation}%
and%
\begin{equation*}
\frac{ds^{\star }}{ds}\sin \theta =-a\tau _{\varphi }.\  \  \  \  \  \  \  \  \  \  \ 
\end{equation*}%
So using Eqn. (4.5) in the last equality, we get 
\begin{equation}
\sin \theta =-a\tau _{\varphi }\cos \theta ,\   \tag{4.6}
\end{equation}%
which is $(i)$. Now write%
\begin{equation*}
\varphi (s)=\psi (s^{\star })-\epsilon \delta (s)N_{\psi }(s).
\end{equation*}%
The above equation implies that
\begin{eqnarray*}
T_{\varphi } &=&\frac{ds^{\star }}{ds}\left[ T_{\psi }-\epsilon \delta
^{^{\prime }}N_{\psi }-\epsilon \delta \left( -\kappa _{\psi
}^{2}T_{\psi }+\frac{\kappa _{\psi }^{\prime }}{\kappa _{\psi }}N_{\psi
}+\tau _{\psi }B_{\psi }\right) \right] \\
\textrm{or}\\
T_{\varphi } &=&\frac{ds^{\star }}{ds}\left[ \left( 1+\epsilon \delta \kappa
_{\psi }^{2}\right) T_{\psi }-\epsilon \left( \delta ^{^{\prime
}}+\delta \frac{\kappa _{\psi }^{\prime }}{\kappa _{\psi }}\right)
N_{\psi }-\epsilon \delta \tau _{\psi }B_{\psi }\right].
\end{eqnarray*}%
Using $\dfrac{\delta ^{\prime }}{\delta }=-\dfrac{\kappa
_{\psi }^{\prime }}{\kappa _{\psi }}$, it follows that%
\begin{equation}
T_{\varphi }=\frac{ds^{\star }}{ds}\left[ \left( 1+\epsilon a\kappa _{\psi
}\right) T_{\psi }-\epsilon a\frac{\tau _{\psi }}{\kappa _{\psi }}%
B_{\psi }\right] .  \tag{4.7}
\end{equation}%
On the other hand, Eqn. (4.4) gives%
\begin{eqnarray*}
T_{\psi }\wedge N_{\psi } &=&\left[ T_{\varphi }\cos \theta +N_{\varphi }%
\frac{\sin \theta }{\kappa _{\varphi }}\right] \wedge N_{\psi }=\left[
T_{\varphi }\cos \theta +N_{\varphi }\frac{\sin \theta }{\kappa _{\varphi }}%
\right] \wedge \left( \epsilon B_{\varphi }\right) \\
&=&\epsilon \left( T_{\varphi }\wedge B_{\varphi }\right) \cos \theta
+\epsilon \left( N_{\varphi }\wedge B_{\varphi }\right) \frac{\sin \theta }{%
\kappa _{\varphi }}=-\epsilon N_{\varphi }\cos \theta +\epsilon T_{\varphi
}\kappa _{\varphi }\sin \theta \\
&=&-\epsilon N_{\varphi }\cos \theta +\epsilon T_{\varphi }\kappa _{\varphi
}^{2}\frac{\sin \theta }{\kappa _{\varphi }}=\epsilon T_{\varphi }\sin \theta
-\epsilon N_{\varphi }\kappa _{\varphi }\cos \theta .\\
\Rightarrow  B_{\psi }&=&\epsilon T_{\varphi }\sin \theta -\epsilon N_{\varphi }\cos \theta .
\end{eqnarray*}%
Using Eqn. (4.5) again, we get%
\begin{equation}
T_{\varphi }=T_{\psi }\cos \theta -\epsilon B_{\psi }\frac{\sin \theta }{%
\kappa _{\psi }}.  \tag{4.8}
\end{equation}%
Taking the vector product of Eqn. (4.5) and Eqn. (4.6), we obtain%
\begin{equation*}
\left( 1+\epsilon a\kappa _{\psi }\right) \sin \theta =a\tau _{\psi }\cos
\theta ,
\end{equation*}%
which is $(ii)$. \ On the other hand from Eqn. (4.7) and Eqn. (4.8), it
follows that
\begin{equation}
\frac{ds^{\star }}{ds}\left( 1+a\epsilon \kappa _{\psi }\right) =\cos
\theta ,  \tag{4.9}
\end{equation}%
\begin{equation}
\frac{ds^{\star }}{ds}(a\tau _{\psi })=\sin \theta .\  \  \  \  \  \  \  \  \  \  
\tag{4.10}
\end{equation}%
Thus inserting Eqn. (4.5) in Eqn. (4.9) and using Eqn. (4.6) in Eqn. (4.10) we
get $(iii)$ and $(iv)$ respectively.
\end{proof}

\begin{theorem}
Let $\psi (s^{\star })\in C^{\infty}$, $s^{\star }\in I^{\star }$ be  a 
Frenet curve with $\tau _{\psi }$ nowhere vanishing and satisfying:
\begin{equation}
\left( 1+a\epsilon \kappa _{\psi }\right) \sin \theta =a\tau _{\psi }\cos
\theta \tag{4.11}
\end{equation}
for some constant $a\neq 0$. Then $\psi (s^{\star })$ is a
$FM$ curve which is non-planar.%
\end{theorem}

\begin{proof}
Define the position vector of curve $\psi (s^{\star })$ as follows
\begin{equation*}
\psi (s^{\star })=\varphi (s)+\frac{a}{\kappa (s)}B_{\varphi }(s).
\end{equation*}%
Let's denote differentiation with respect to $s$ by a dash, we have%
\begin{equation*}
\psi ^{\prime }(s^{\ast })=T_{\varphi }-\frac{a\tau _{\varphi }}{\kappa
_{\varphi }}N_{\varphi }.
\end{equation*}%
Since $\tau _{\varphi }\neq 0$, we see that $\psi (s^{\star })$ is a $%
C^{\infty }$ regular curve. Suppose all the quantities of $\psi (s^{\star })$
be marked by a asterisk. Then%
\begin{equation*}
T_{\psi }\frac{ds^{\star }}{ds}=T_{\varphi }-\frac{a\tau _{\varphi }}{\kappa
_{\varphi }}N_{\varphi }.
\end{equation*}%
Hence, we have%
\begin{equation*}
\frac{ds^{\star }}{ds}=\sqrt{1-a^{2}\tau _{\varphi }^{2}}.
\end{equation*}%
Using Eqn. (4.11), we get 
\begin{equation*}
T_{\psi }=T_{\varphi }\cos \theta +N_{\varphi }\frac{\sin \theta }{\kappa
_{\varphi }},
\end{equation*}%
notice that from Eqn. (4.11), we have $\sin \theta \neq 0$. Therefore%
\begin{equation*}
\frac{dT_{\psi }}{ds^{\star }}\frac{ds^{\star }}{ds}=N_{\varphi }\cos \theta
+\frac{\sin \theta }{\kappa _{\varphi }}\left( -\kappa _{\varphi
}^{2}T_{\varphi }+\frac{\kappa _{\varphi }^{\prime }}{\kappa _{\varphi }}%
N_{\varphi }+\tau _{\varphi }B_{\varphi }\right) -\frac{\kappa _{\varphi
}^{\prime }}{\kappa _{\varphi }^{2}}\sin \theta N_{\varphi }
\end{equation*}%
or%
\begin{equation}
\kappa _{\psi }\frac{ds^{\star }}{ds}\frac{N_{\psi }}{\kappa _{\psi }}%
=-\left( \kappa _{\varphi }\sin \theta \right) T_{\varphi }+N_{\varphi }\cos
\theta +\allowbreak \left( \tau _{\varphi }\sin \theta \right) \frac{%
B_{\varphi }}{\kappa _{\varphi }}.  \tag{4.12}
\end{equation}%
Now define $\frac{N_{\psi }}{\kappa _{\psi }}=\epsilon \frac{B_{\varphi }}{%
\kappa _{\varphi }}$. Then from Eqn. (4.12), we get%
\begin{equation*}
\kappa _{\psi }=\epsilon \frac{ds}{ds^{\star }}\tau _{\varphi }\sin \theta .
\end{equation*}%
These are $C^{\infty }$ functions of $s$ (and hence of $s^{\star }$), and%
\begin{equation*}
\frac{dT_{\psi }}{ds^{\star }}=N_{\psi }.
\end{equation*}%
Again define $B_{\psi }=T_{\psi }\wedge B_{\varphi }$ and 
\begin{eqnarray*}
&&\left \langle \frac{dB_{\psi }}{ds^{\star }},N_{\psi }\right \rangle
=\left \langle -\tau _{\psi }N_{\psi }+\frac{\kappa _{\psi }^{\prime }}{%
\kappa _{\psi }}B_{\psi },N_{\psi }\right \rangle =-\tau _{\psi }\left
\langle N_{\psi },N_{\psi }\right \rangle =-\tau _{\psi }\kappa _{\psi
}^{2}\text{ }\\
&&\textrm{or}\\
&&\tau _{\psi }=-\frac{\left \langle \frac{dB_{\psi }}{ds^{\star }},N_{\psi
}\right \rangle }{\kappa _{\psi }^{2}}.
\end{eqnarray*}%
These are also of class $C^{\infty }$ on $I^{\star }$. It is then trivial to
check that with the functions $\kappa _{\psi },\tau _{\psi }$ and with modified frame $\left \{ T_{\psi },N_{\psi },B_{\psi
}\right\} $ , the curve 
$\psi (s^{\star })$ becomes a $C^{\infty }$ Frenet curve. But $B_{\varphi }$
and $N_{\psi }$ lie on the joining line of corresponding points of $\varphi (s)$
and $\psi (s^{\star })$. Thus $\varphi(s)$ is a $FM$ conjugate of a $FM$ curve $\psi (s^{\star })$.
\end{proof}

\begin{lemma}\label{lem 2}
A necessary and sufficient condition for a regular curve $%
\psi \in C^{\infty } $ to be a $FM$ curve with a $FM$ conjugate is that $\psi $ is either a line or a non-planar circular helix.
\end{lemma}

\begin{proof}
$\Rightarrow :$ Suppose a line $\varphi$  be a $FM$ conjugate of $\psi$. This implies $\kappa _{\varphi }=0.$ Using Lemma \ref{lem1}, $(iii)$ and $(i)$, $(ii)$,
we have%
\begin{equation}
\cos ^{2}\theta =1+a\epsilon \kappa _{\psi }  \tag{4.13}
\end{equation}%
and then%
\begin{equation}
\cos ^{2}\theta \sin \theta =a\tau _{\psi }\cos \theta ,~\  \  \  \  \  
\tag{4.14}
\end{equation}%
\begin{equation}
\cos \theta =-a\tau _{\varphi }\sin \theta .\  \  \  \  \  \  \  \  \  \  \  \  \  
\tag{4.15}
\end{equation}%
From Eqn. (4.15), it follows that $\cos \theta \neq 0$. Hence Eqn. (4.14) is
proportional to%
\begin{equation}
a\tau _{\psi }=\cos \theta \sin \theta .\  \  \  \  \  \  \  \  \  \  \  \  \  \  \  
\tag{4.16}
\end{equation}

\textbf{Case 1.} $\sin \theta =0$. Then $\cos \theta =\pm 1$, so that (4.13)
implies $\kappa _{\psi }=0$, and $\psi $ is a line. We  also note that
Eqn. (4.16) implies that $\tau _{\psi }=0$.

\textbf{Case 2.} $\sin \theta \neq 0$. Then $\cos \theta \neq \pm 1$, and
Eqns. (4.13),(4.16) imply that $\kappa _{\psi },\tau _{\psi }$ are non-vanishing
constants, and $\psi $ is a non-planar circular helix.

$\Leftarrow :$ \ Suppose $\psi $ be a non-planar circular helix given by $$\psi =\left(
r\cos t,r\sin t,bt\right) ,\left( a,b\in R^{+}\right), $$ where $t=\dfrac{s}{%
\sqrt{r^{2}+b^{2}}}.$\
\newline
 We may write
\begin{equation*}
\psi ^{^{\prime }}(s)=T_{\psi }=\frac{r}{\sqrt{r^{2}+b^{2}}}\left( -\sin
t,\cos t,k\right)
\end{equation*}%
and 
\begin{equation*}
\psi ^{^{\prime \prime }}(s)=N_{\psi }=\frac{-r}{r^{2}+b^{2}}\left( \cos
t,\sin t,0\right).
\end{equation*}
\begin{equation*}\Rightarrow \kappa _{\psi }=\frac{r}{r^{2}+b^{2}}\text{
and }-r\frac{N_{\psi }}{\kappa _{\psi }}=r\left( \cos t,\sin t,0\right).
\end{equation*}%
Then the curve $\psi $ with%
\begin{eqnarray*}
\psi &=&\left( r\cos t,r\sin t,bt\right) =r\left( \cos t,\sin t,0\right)
+\left( 0,0,bt\right) \\
&=&-\frac{r}{\kappa _{\psi }}N_{\psi }+\varphi =\varphi -\frac{r}{\kappa
_{\psi }}\left( \frac{\kappa _{\psi }}{\kappa _{\varphi }}B_{\varphi
}\right) =\varphi -\frac{r}{\kappa _{\varphi }}B_{\varphi }
\end{eqnarray*}%
or putting $\delta =-\dfrac{r}{\kappa _{\varphi }}$ 
\begin{equation*}
\psi =\varphi +\delta B_{\varphi }
\end{equation*}%
will be a line along the $z-$axis, and can be turned into a $FM$ conjugate of $%
\psi $ by defining $\frac{N_{\psi }}{\kappa _{\psi }}$ to be equal to $%
\frac{B_{\varphi }}{\kappa _{\varphi }}$.
\end{proof}

\begin{theorem}
\bigskip Let $\psi (s^{\star })\in C^{\infty }$ be a plane Frenet curve with zero torsion and the curvature  being bounded above or bounded below.
Then $\psi $ is a $FM$ curve, with $FM$ conjugates which are
curves in plane.
\end{theorem}

\begin{proof}
Suppose $\psi $ be a curve satisfying the given conditions. Thent $\kappa _{\psi }<-\frac{1}{a}$ or $%
\kappa _{\psi }>-\frac{1}{a}$ on $I$ for some constant $a\neq 0$. For such $a$, let $\varphi$ be the curve with position vector 
\begin{equation*}
\varphi =\psi -\varepsilon \delta N_{\psi }.
\end{equation*}%
Differentiating last equality with respect to $s^{\ast }$ and
considering equality $\delta ^{^{\prime }}+\delta \dfrac{\kappa _{\varphi
}^{^{\prime }}}{\kappa _{\varphi }}=0$ and $\delta =\dfrac{a}{\kappa
_{\varphi }}$ $\left( a\in R\right),$ we get 
\begin{equation*}
T_{\varphi }\frac{ds}{ds^{\ast }}=T_{\psi }-\varepsilon \delta ^{^{\prime
}}N_{\psi }-\varepsilon \delta \left( -\kappa _{\psi }^{2}T_{\psi }+%
\frac{\kappa _{\psi }^{^{\prime }}}{\kappa _{\psi }}N_{\psi }\right)
=\left( 1+\varepsilon \delta \kappa _{\psi }^{2}\right) T_{\psi
}-\varepsilon \left( \delta ^{^{\prime }}+\delta \frac{\kappa _{\psi
}^{^{\prime }}}{\kappa _{\psi }}\right) N_{\psi }
\end{equation*}%
or
\begin{equation*}
T_{\varphi }=\frac{ds^{\ast }}{ds}\left( 1+\varepsilon a\kappa _{\psi
}\right) T_{\psi }.
\end{equation*}%
Since $\frac{ds^{\ast }}{ds}\left( 1+\varepsilon a\kappa _{\psi }\right)
\neq 0$, $\varphi $ is a $C^{\infty }$ regular curve, and $T_{\varphi
}=T_{\psi }$. It is then easy to verify that $\varphi $
is a $FM$ conjugate of $\psi $.
\end{proof}

\section{Weakened Mannheim curves}

\begin{definition}
Let $\mathcal{D}$ be a subset of a topological space $\mathcal{X}$. If a function: $\mathcal{X} \mapsto \mathcal{Y}$ is constant for each component of $\mathcal{D}$, we say that function is $\mathcal{D}$-piecewise constant [2].
\end{definition}

\begin{lemma}\label{lem 3}
Suppose $\mathcal{D}$ be an open subset of a proper interval $\mathcal{X}$ of the real line. A necessary and sufficient condition for every $\mathcal{D}$-piecewise constant real continuous function on $\mathcal{X}$ to be constant is that there is an empty dense-in-itself kernel of $\mathcal{X} \backslash \mathcal{D}$ [2].
\end{lemma}

We remark, however, if $\mathcal{D}$ is dense in $\mathcal{X}$, any $\mathcal{D}$%
-piecewise constant, $C^1$ real function on $\mathcal{X}$ is constant, even if $\mathcal{D}$ has a
non-empty dense-in-itself kernel.

\begin{theorem}
A $WM$ curve for which $M$ and $N$(defined in def. \ref{def 1}) have void dense-in-itself kernels is a $%
FM$ curve.
\end{theorem}

\begin{proof}
Let $\varphi(s)$, $s\in I$ be a $WM$ conjugate of $WM$ curve $\psi (s^{\star }),s^{\star }\in I^{\star }$. From the definition of $\varphi (s)$ and $\psi (s^{\star })$, it follows that both the curves have a $C^\infty$ family of the tangent vectors $T_{\psi }(s^{\star })$, $T_{\varphi
}(s)$. Let%
\begin{equation}
\psi (s)=\psi (\rho (s))=\varphi (s)+\delta (s)B_{\varphi }(s),  \tag{5.1}
\end{equation}%
where $B_{\varphi }(s)$ is a vector function and $ \delta
(s)\geq 0$, $\forall s\in I$. Let $\mathcal{D}=I\backslash N$, $\mathcal{D}^{\star }=I^{\star
}\backslash \rho (M)$. Then $s^{\star }(s)\in $ $C^{\infty }$ on $D^{\star
}.$

\textbf{Step 1.} 
Proving $\delta =\frac{a}{\kappa _{\varphi }(s)}.$ Since $%
\delta(=\left \Vert \psi (s)-\varphi (s)\right \Vert)$ is $C^{\infty }$ on $I$ and nowhere zero on every interval of $\mathcal{D}$. Let $\mathcal{X}$ be any component of $\mathcal{P}:=\left \{ s\in I:\delta (s)\neq 0\right \} $. Hence both $\mathcal{P}$ and $\mathcal{X}$ are open in $I$. Consider $\mathcal{L}$ as a component interval of $\mathcal{X}\cap \mathcal{D}$. Then, $\delta (s)$ and $%
B_{\varphi }(s)$ are of class $C^{\infty }$ on $\mathcal{L}$, and from Eqn. (5.1), we have%
\begin{equation*}
\psi ^{\prime }(s)=T_{\varphi }+\delta ^{\prime }(s)B_{\varphi }(s)+\delta
(s)B_{\varphi }^{\prime }(s).
\end{equation*}%
By the definition of a $WM$ curve, we know that $\left \langle T_{\varphi
},B_{\varphi }(s)\right \rangle =0=$ $\left \langle \psi ^{\prime }(s^{\star
}),B_{\varphi }(s)\right \rangle $. Hence, using$\left \langle
B_{\varphi }^{\prime }(s),B_{\varphi }(s)\right \rangle =0$, we get
\begin{equation*}
0=\left( \delta ^{\prime }(s)+\delta (s)\frac{\kappa _{\varphi }^{\prime
}(s)}{\kappa _{\varphi }(s)}\right) \left \langle B_{\varphi }(s),B_{\varphi
}(s)\right \rangle .
\end{equation*}%
Therefore $\delta =\frac{a}{\kappa _{\varphi }}$ on $\mathcal{L}$, where $a$ is
constant. Thus $\delta $ is not constant on each interval of the set $%
\mathcal{X}\cap \mathcal{D}$. But by the given conditions $\mathcal{X}\backslash \mathcal{D}$ has void dense-in-itself
kernel. It follows from Lemma \ref{lem 3} that $\delta $ is not constant (and
non-zero) on $\mathcal{X}$. As $\delta $ is continuous on $I$, $\mathcal{X}$ should be closed
in $I$. But in $I$, $\mathcal{X}$ is also open. Hence, by connectedness, 
$\mathcal{X}=I$, i.e., $\delta $ is not constant on $I$.

\textbf{Step 2.} Existence of frames 
\begin{equation*}
\left \{ T_{\varphi }(s),N_{\varphi }(s),N_{\varphi }(s)\right \} ,\left \{
T_{\psi }(s^{\star }),N_{\psi }(s^{\star }),B_{\psi }(s^{\star })\right \}
\end{equation*}%
which are the modified orthogonal frames for $\varphi (s)$ on $\mathcal{D}$ and $\psi (s^{\star })$
on $\mathcal{D}^{\star }$, respectively. Since $\delta =\frac{a}{\kappa _{\varphi }%
}$ is a non-zero function, it follows from Eqn. (5.1) that $B_{\varphi }(s)$ is
continuous on $I$ and is of class $C^{\infty }$ on $\mathcal{D}$, and orthogonal to $%
T_{\varphi }(s)$. Define $B_{\varphi }(s)=T_{\varphi }(s)\wedge N_{\varphi
}(s)$, then $\left \{ T_{\varphi }(s),N_{\varphi }(s),B_{\varphi }(s)\right \} $
forms a right-handed modified orthonormal frame for $\varphi (s)$ which is of class $C^\infty$ on $\mathcal{D}$ and
continuous on $I$.

From the definition of $WM$ curve, it follows that there exists a scalar
function $\kappa _{\psi }(s^{\star })$ such that $T_{\psi }^{\prime
}(s^{\star })=\kappa _{\psi }(s^{\star })N_{\psi }(s^{\star })$ on $I^{\star }$. Thus, $\left \langle
T_{\psi }^{\prime }(s^{\star }),N_{\psi }(s^{\star })\right \rangle
=\kappa _{\psi }^{2}(s^{\star })$ is continuous on $I^{\star }$ and of class $%
C^{\infty }$ on $\mathcal{D}^{\star }$. Hence the first Frenet formula holds on 
$\mathcal{D}^{\star }$. It is now easy to show that there exists a function $\tau _{\varphi }(s)$ of class $C^\infty$ on $\mathcal{D}$ such that the Frenet
formulas hold. Thus $\left \{ T_{\varphi }(s),N_{\varphi }(s),B_{\varphi
}(s)\right \} $ is a modified orthogonal frame for	 $\varphi (s)$ on $\mathcal{D}$.

Similarly, there exists a modified orthogonal frame $\left \{T_{\psi }(s^{\star }),N_{\psi }(s^{\star }),B_{\psi }(s^{\star
})\right
\} $ for $\psi (s^{\star })$, which is continuous on $I^{\star }$
and is a Frenet frame for $\psi (s^{\star })$ on $\mathcal{D}^{\star }$. Moreover, we
may choose%
\begin{equation*}
B_{\varphi }(s)=N_{\psi }(\rho (s)).
\end{equation*}%
\textbf{Step 3.} Show that $N=\phi $, $M=\phi.$ 

Noticing that on $\mathcal{D}$, we have%
\begin{equation*}
\left \langle T_{\psi },T_{\varphi }\right \rangle ^{\prime }=\left \langle
N_{\psi }\frac{ds^{\star }}{ds},T_{\varphi }\right \rangle +\left \langle
T_{\psi },\kappa _{\varphi }N_{\varphi }\right \rangle =0,
\end{equation*}%
so that on each component of $\mathcal{D}$, $\left \langle T_{\psi },T_{\varphi }\right \rangle $ is constant and hence on $I$ by Lemma \ref{lem 3}. Thus, there
exists a angle $\theta $ such that%
\begin{equation*}
T_{\psi }=T_{\varphi }\cos \theta +N_{\varphi }\sin \theta .
\end{equation*}%
Further,%
\begin{equation*}
B_{\varphi }(s)=N_{\psi }(\rho (s))
\end{equation*}%
and so 
\begin{equation*}
B_{\psi }(s^{\star })=-T_{\varphi }\sin \theta +\frac{N_{\varphi }}{\kappa
_{\varphi }}\cos \theta .
\end{equation*}%
Hence 
$\left \{ T_{\psi }(s^{\star }),N_{\psi }(s^{\star }),B_{\varphi}(s)\right \} $ are also of class $C^{\infty }$ on $\mathcal{D}$. On the other hand, with respect to $s^{\star }$ on 
$\mathcal{D}^{\star }$, $%
\left \{ T_{\psi }(s^{\star }),N_{\psi }(s^{\star }),B_{\psi }(s^{\star
})\right \} $ are of class $C^{\infty }$. Writing Eqn. (5.1) in the form%
\begin{equation*}
\varphi =\psi -\frac{a}{\kappa _{\psi }}N_{\psi }\text{ or }\varphi =\psi
-\delta N_{\psi }
\end{equation*}%
and differentiating with respect to $s$ on $\mathcal{D}\cap \rho ^{-1}(\mathcal{D}^{\star })$,
we have%
\begin{equation}
T_{\varphi }=\frac{ds^{\star }}{ds}\left[ \left( 1+a\kappa _{\psi }\right)
T_{\psi }-\frac{a\tau _{\psi }}{\kappa _{\psi }}B_{\psi }\right] . 
\tag{5.2}
\end{equation}%
But%
\begin{equation*}
T_{\varphi }=T_{\psi }\cos \theta -B_{\psi }\sin \theta .
\end{equation*}%
Hence%
\begin{equation}
\frac{ds^{\star }}{ds}\left( 1+a\kappa _{\psi }\right) =\cos \theta \text{
and }\frac{a\tau _{\psi }}{\kappa _{\psi }}=-\sin \theta .  \tag{5.3}
\end{equation}%
Since $\kappa _{\psi }^{2}(s^{\star })=\left \langle T_{\psi }^{\prime
},N_{\psi }\right \rangle $ is continuous on $I^{\star }$ and $%
\rho ^{-1}(\mathcal{D}^{\star })$ is dense, it follows by continuity that Eqn. (5.3)
holds throughout $\mathcal{D}$.

\textbf{Case 1.} $\cos \theta \neq 0$. Then Eqn. (5.3) implies $\frac{%
ds^{\star }}{ds}\neq 0$ on $\mathcal{D}$. Hence $M=\phi $. Similarly $%
N=\phi $.

\textbf{Case 2.} $\cos \theta =0$. Then%
\begin{equation}
T_{\psi }=\pm \frac{N_{\varphi }}{\kappa _{\varphi }}.  \tag{5.4}
\end{equation}%
Taking derivative of Eqn. (5.1) with respect to $s$ in $\mathcal{D}$, we get%
\begin{equation}
T_{\psi }\frac{ds^{\star }}{ds}=\pm \frac{a\tau _{\varphi }}{\kappa _{\varphi
}}N_{\varphi }.  \tag{5.5}
\end{equation}%
Hence using Eqn. (5.4) in Eqn. (5.5), we have%
\begin{equation*}
\frac{ds^{\star }}{ds}=\pm a\tau _{\varphi }.
\end{equation*}

Therefore, we get 
\begin{equation*}
\tau _{\varphi }=\pm \frac{1}{a}\frac{ds^{\star }}{ds}
\end{equation*}%
and so also on $I$, by Lemma \ref{lem 3}. It follows that $\tau _{\varphi }$ is nowhere
zero on $I$. Consequently $\psi (s^{\star })=\varphi (s)+\delta
(s)B_{\varphi }(s)$ is of class $C^{\infty }$ on $I^{\star }$. Hence $%
N=\phi$. Similarly $M=\phi $.
\end{proof}


\begin{thebibliography}{9}
\bibitem{1} Karacan MK. Weakened Mannheim curves. Int J Phys Sci 2011; 6: 4700-4705.

\bibitem{2} Lai HF. Weakened Bertrand curves. Tohoku Math Journ 1967; 19: 141-155.

\bibitem{3} Liu H, Wang F. Mannheim partner curves in 3-space. J Geom 2008; 88: 120-126.

\bibitem{4} Orbay K, Kasap E. On Mannheim partner curves in $E^{3}$. Int J Phys Sci 2009; 4: 261-264.

\bibitem{5} \"{O}ztekin HB, Erg\"{u}t M. Null Mannheim curves in the minkowski 3-space $\mathbb{E}_1^3$. Turk J Math 2011; 35: 107-114.

\bibitem{6} \"{O}ztekin HB. Weakened Bertrand curves in the Galilean Space $G_{3}$. J Adv Math Studies 2009; 2: 69-76.

\bibitem{7} Sasai T. The Fundamental Theorem of Analytic Space Curves and Apparent Singularities of Fuchsian Differential Equation\textit{s}. Tohoku Math Journ 1984; 36: 17-24.

\bibitem{8} Tun\c{c}er Y, Karacan MK. Non-null weakened Mannheim curves in Minkowski 3-space. An \c{S}tiin\c{t} Al I Cuza In\c{s}i Mat 2017; LXIII: 403-412.
\end{thebibliography}
\end{document}